\documentclass[reqno]{amsart}
\usepackage{pgf,tikz}
\usetikzlibrary{arrows}

\usepackage{undertilde}
\usepackage{mathtools}
\usepackage{amsthm}
\usepackage{amsmath}
\usepackage{verbatim}
\usepackage{amssymb}
\usepackage{graphicx}
\usepackage{latexsym}
\usepackage[mathscr]{euscript}

\newtheorem{theorem}{Theorem}[subsection]

\usepackage[usenames,dvipsnames]{}

\newtheorem{lemma}[theorem]{Lemma}

\newtheorem{proposition}[theorem]{Proposition}
\newtheorem{corollary}[theorem]{Corollary}
\newtheorem{claim}[theorem]{Claim}
\newtheorem{fact}[theorem]{Fact}

\theoremstyle{definition}
\newtheorem{definition}[theorem]{Definition}
\newtheorem{notation}[theorem]{Notation}

\theoremstyle{remark}
\newtheorem{remark}[theorem]{Remark}
\newtheorem{example}[theorem]{Example}

\newcommand{\px}[2]{{#1}{\upharpoonright_{#2}}}

\newcommand{\cdf}{\mathrm{\textbf{cdf}} }
\newcommand{\Lip}{\mathrm{\textbf{Lip}} }

\newcommand{\zeroone}{[0,1] } 
\newcommand{\zn}{[0,1]^n }

\newcommand{\I}{\mathbb I }

\newcommand{\arpi}[2]{\Pi_{#2}^{#1}}
\newcommand{\arsigma}[2]{\Sigma_{#2}^{#1}}

\newcommand{\fcantor}{2^{<\omega}}

\newcommand{\cylinder}[1]{\left[ #1 \right]}

\newcommand{\real}{\mathbb{R}}

\newcommand{\lm}[1]{\lambda\left( #1 \right)}

\usepackage{ifpdf}

\ifpdf
\usepackage[bookmarks,colorlinks=true]{hyperref}
\else
\usepackage{hyperref}
\fi

\usepackage[UKenglish]{babel}

\newcommand{\A}{\mathcal{A}}

\newcommand{\D}{\mathcal{D}}

\newcommand{\N}{\mathbb{N}}

\renewcommand{\P}{\mathcal{P}}

\newcommand{\Q}{\mathbb{Q}}
\newcommand{\R}{\mathbb{R}}

\newcommand{\NN}{{\mathbb{N}}}

\renewcommand{\P}{\mathcal P}
\newcommand{\bi}{\begin{itemize}}
\newcommand{\ei}{\end{itemize}}
\newcommand{\bc}{\begin{center}}
\newcommand{\ec}{\end{center}}

\renewcommand{\P}{\mathcal P}

\newcommand{\n}{\noindent}

\newcommand{\seq}[1]{(#1_i)_{i\in\NN}}

\usepackage[normalem]{ulem}

\begin{document}

\title{Computable randomness and monotonicity}

\author{Alex Galicki}


\maketitle

\begin{abstract}
We show that $z\in\R^n$ is computably random if and only if every computable monotone function on $\R^n$ is differentiable at $z$.
\end{abstract}


\section{Introduction}

Our main result is concerned with differentiability of monotone functions of several variables. Monotone functions are closely related to Lipschitz functions and they play a prominent role in variational analysis (see \cite{Rockafellar.Wets:97}) and in the theory of optimal transport (see \cite{Villani:03}). It is known that on the unit interval differentiability of computable monotone functions is equivalent to computable randomness. 

\begin{theorem}[Theorem 4.1 in \cite{Brattka.Miller.ea:nd}]\label{t_nies1} A real $z$ is computably random $\iff$ every computable nondecreasing function $f:\zeroone\to\real$ is differentiable at $z$.
\end{theorem}

\noindent We will prove the following generalization of the above result.

\begin{theorem}\label{t_monotone_full} Let $n\ge 1$. $z\in\R^n$ is computably random $\iff$  every computable monotone function $f:\R^n\to\R^n$ is differentiable at $z$.
\end{theorem}

\n The proof has two distinct parts. The  $\Rightarrow$ implication is proven in Section \ref{forward_direction} using an effective form of the Rademacher Theorem and geometric properties of monotone functions. The converse implication
is proven in Section \ref{back_direction} and uses results from optimal transport.

\section{Differentiability of  computable monotone functions\\ from $\R^n$ to $\R^n$}\label{forward_direction}

\n Let $f:\R^n\to\R^n$  be a function. We say $f$\ is \emph{monotone} if $$\langle f(x)-f(y),x-y\rangle\ge0\text{ for all }x,y\in\R^n.$$

\n As we will see later in this section, monotone functions are very closely related to Lipschitz functions.

In non-effective setting, a.e. differentiability of monotone function from $\R^n$ to $\R^n$ has been proven by Mignot \cite{Mignot:76}, who used Rademacher's Theorem and a fact about monotone functions discovered by Minty \cite{Minty:62}.

In this section we will show that computable randomness implies differentiability of computable monotone real functions of several variables and our proof follows the same path - using the effective form or Rademacher's Theorem proven in the previous section and the following correspondence observed by Minty. 
\subsection{Minty parameterization and overview of the proof}

Minty showed that the so called Cayley transformation  

\[\Phi:\R^n\times\R^n\to \R^n\times\R^n \text{ defined by } \Phi(x,y)=\frac{1}{\sqrt2}(y+x,y-x)\]   transforms the graph of a monotone function into a graph of a  graph of a 1-Lipschitz function. Note that when $n=1$ this is a clockwise rotation of $\pi/4$. We will rely on the following consequence of the above fact. 

\begin{proposition}[cf. Proposition 1.2 in \cite{Alberti.Ambrosio:99}]\label{monotone_lip_lemma}
Let $u:\R^n\to\R^n$ be monotone. Then $(u+I)$ and $(u+I)^{-1}$ are monotone and $(u+I)^{-1}$ is 1-Lipschitz.
\end{proposition}

\begin{proposition}[cf. Theorem 12.65 in \cite{Rockafellar.Wets:97}]\label{inverse_diff_proposition}
Let $u:\R^n\to\R^n$ be a continuous monotone function. Let $z\in\R^n$ and define $f=(u+I)^{-1}$ and $\hat z=u(z)+z.$ The following two are equivalent:
\begin{enumerate}
\item $u$ is differentiable at $z$, and
\item  $f$ is differentiable at $\hat z$ and $f'(z)$ is invertible.
\end{enumerate}
\end{proposition}

A good exposition of classical results related to this area can be found in \cite{Alberti.Ambrosio:99} and \cite{Rockafellar.Wets:97}.

\n Now we are ready to explain our proof.

\vspace{5pt}
\n\emph{Overview of the proof} 

\n Let $u:\R^n\to\R^n$ be  a monotone computable function and let $z\in\zn$ be computably random. Then $g=u+I$ is monotone and computable and $f=g^{-1}$ is 1-Lipschitz and computable.
If we can show that $g(z)$ is computably random, then $f$ is differentiable at $g(z)$. By Proposition \ref{inverse_diff_proposition}, if the derivative of $f$ at $g(z)$ is invertible, then $g$ is differentiable at $z$.

From the above description, it is clear that we require the following two ingredients to complete the proof:

\begin{enumerate}
\item[(preservation property)] we need to show that $g(z)$ is computably random when $z$ is, and
that\item[(singularity property)] computable randomness of $g(z)$ implies that $f'(g(z))$ is invertible. 
\end{enumerate}

\n In the following two subsections we will prove both of the above.

\subsection{Another preservation property}\label{preservation_lemma_subsection}

To prove the preservation property mentioned in the previous subsection, we require some terminology and notation from \cite{Rute:13}.

Firstly, we need to extend the notion of computable randomness to $\R^n$: we say $z\in \R^n$ is computably random if its binary expansion (or, equivalently, its fractional part) is computably random. When $z\in\zn$, this characterisation is equivalent to the Definition \ref{cr_definition}. Otherwise, when $z\notin \zn$, this characterisation is equivalent to computable randomness on some computable translation of the unit cube equipped with the usual Lebesgue measure. 

\begin{notation}

For every $n\ge 1$, let $\A_n$ be some fixed a.e. decidable cell decomposition of $\zn$. For the sake of simplifying the notation, in the rest of this section, for all $n\ge1$ and all  $\sigma \in \fcantor$, we denote the cell $[\sigma]_{\A_n}$ by $[\sigma]$.

\end{notation}

\begin{definition}

A \emph{Martin-L\"of test} is a uniformly computable sequence $\seq U$ of $\arsigma01$ subsets of $\zn$ such that $\lm{U_i}\le 2^{-i}$ for all $i$. We say $\seq U$ \emph{covers} $z\in\zn$ if $z\in\bigcap_i U_i$.

\n We say a Martin-L\"of test $\seq U$ is \emph{bounded} if there is a computable measure $\nu:\fcantor\to[0,\infty)$ satisfying \[\lm {U_i\cap [\sigma]}\le 2^{-i}\nu(\sigma)\] for all $i\in\N$ and $ \sigma\in\fcantor$.  

\end{definition}

\n We require the following characterisation of computable randomness in the unit cube due to Rute:
\begin{proposition}[cf. Theorem 5.3 in \cite{Rute:13}]\label{cr_characterisation_ml}
Let $z\in\zn$. The following two are equivalent:
\begin{enumerate}
\item $z$ is not computably random, and
\item either $z$ is an unrepresented point, or there is a bounded Martin-L\"of test $\seq U$ that covers $z$.
\end{enumerate}
\end{proposition}

\begin{remark}
For our considerations it is sufficient to know that if $z$ is an unrepresented point, then it is not weakly random. 
\end{remark}

\n We are now in position to state and to prove the required preservation property for computable randomness.

\begin{lemma}\label{no_random_lemma}
Let $f:\R^n\to\R^n$ be a computable injective Lipschitz function and suppose $z\in \R^n$ is not computably random. Then $f(z)$ is not computably random either.

\begin{proof}

Without loss of generality we assume $z\in\zn$, $f(z)\in\zn$ and  $\Lip(f)\le1$ (otherwise we may consider $\hat f(x)=A\cdot f(x)+B$ for some suitable computable $A$ and $B$).

Firstly, let's assume that $z$ is an unrepresented point. Let $P\subset\zn$ be a $\arpi01$ null set with $z \in P$. Then $f(z)\in f(P)\cap\zn$ and since $f(P)\cap\zn$ is also a $\arpi01$ null set, $f(z)$ is not weakly random.

Let $\seq V$ be a bounded Martin-L\"of test with $z\in \bigcap_i V_i$ and let $\nu$ be a computable measure such that $\lm{V_i\cap \cylinder \sigma}\le 2^{-i}\nu (\sigma)$ for all $i,\sigma$. 

Define $U_i=f(V_i)\cap \zn$ for all $i$. Since $\lm{U_i}\le \lm{V_i}$ (see Lemma 3.10.12 in \cite{Bogachev.vol1:07}) and $f$ is injective, $\seq U$ is a Martin-L\"of test. 

Define $\nu_f=\nu\circ f^{-1}$. It is a computable measure and for all $i,\sigma$ we have 
\begin{align*}
\lm{U_i\cap \cylinder \sigma}=\lm{f(V_i)\cap \cylinder \sigma}=\\\lm{f(V_i\cap f^{-1}(\cylinder \sigma))} \le  2^{-i}\nu \left(f^{-1}(\cylinder \sigma)\right)=2^{-i}\nu_f(\sigma).
\end{align*}

It follows that $\seq U$ is a bounded Martin-L\"of test that covers $f(z)$ and thus $f(z)$ is not computably random.

\end{proof}
\end{lemma}

\subsection{Singularity property}\label{subsub33}

The main result in this subsection, Theorem \ref{critical_point_theorem}, can be seen as an effective version of Sard's Theorem for Lipschitz function. Its classical version, proven by Mignot (\cite{Mignot:76}, also see Theorem 9.65 in \cite{Rockafellar.Wets:97}), states that for a Lipschitz function $f:\R^n\to\R^n$, the set of its critical values is a null-set. 

\begin{lemma}\label{critical_value_lemma}
Let $f:\R^n\to\R^n$ be a Lipschitz function. Suppose $z\in\R^n$ is such that $f'(z)$ is singular. Then for every $\epsilon>0$, there exists an open neighbourhood of $z$, $O_\epsilon$ such that $\lm{f(O_\epsilon)}\le \epsilon\lm{O_\epsilon}$.
\begin{proof}

Fix $\epsilon>0$ and let $k=\Lip(f)$.

Define $\epsilon'=\frac{\epsilon}{k^{n-1}2^n(\sqrt n)^n}$. Since $f$\ is differentiable at $z$, there exists $\delta>0$ such that 
\begin{align}\label{eq331.1}
\left|f(x)-f(z)-f'(z)(x-z)\right|\le \epsilon'|x-z|
\end{align}
for all $x\in \R^n$ with $|x-z|\le \delta$. There is an open $n$-cube $C$ with side length equal to $s=\frac\delta{\sqrt n}$ such that $z\in C$ and \ref{eq331.1} holds for all $x\in C$. 

Let $L$ be the mapping defined by $L(x)=f(z)+f'(z)(x-z)$. Since $f'(z)$ is singular, $L$\ is not onto and its range is contained in some hyperplane $H$.

As a consequence of \ref{eq331.1} we have $\left|f(x)-L(x)\right|\le \epsilon'\delta$ for all $x\in C$. Thus, $f(C)\subseteq L(C)+[-\epsilon'\delta,\epsilon'\delta]^n$. Since $L$ is a $k$-Lipschitz mapping, the image of $C$\ under $L$ lies in the intersection of $H$ with a closed ball with radius $k\delta$ centered at $f(z)$. Then $L(C)$ is contained in a rotated $(n-1)$-dimensional cube of side $2k\delta$. This shows that $f(C)$ lies in a rotated  box  $\hat C$ with 

\[
\lm{\hat C}=(2k\delta)^{n-1}2\epsilon'\delta=2(2k)^{n-1}\epsilon'(\sqrt n)^n\left(\frac{\delta}{\sqrt n}\right)^n=\epsilon\lm{C}.
\]

\end{proof}
\end{lemma}

\begin{theorem}\label{critical_point_theorem}
Let $f:\R^n\to\R^n$ be a computable Lipschitz function and let $z\in\R^n$. If $f(z)$ is computably random, then $f'(z)$ is not singular.
\begin{proof}
Without loss of generality we may assume $f(z)\in\zn$ and $\zn\subseteq f(\zn)$. The proof is by contraposition. Suppose $f'(z)=0$.

Let $\nu=\lambda\circ f^{-1}$ and for every $i\in \N$, define $V_i\subset \zn$ as the union of all $[\sigma]$ such that $\lm\sigma\le 2^{-i}\nu(\sigma)$. Note that $\lm{V_i}\le 2^{-i}$ and for every $\tau,$ 
\[
\lm{V_i\cap [\tau]}=\sum_{[\eta]\subseteq [\tau]\cap V_i }\lm{\eta}\le 2^{-i}\sum_{[\eta]\subseteq [\tau]\cap V_i }\nu(\eta)\le2^{-i}\nu(\tau).
\]

Thus $\seq V$ is a bounded Martin-L\"of test and, by Lemma \ref{critical_value_lemma}, it covers $f(z)$.
\end{proof}
\end{theorem}

\subsection{Main result}

\n We are now ready to formulate and prove the main result of this section.

\begin{theorem}\label{theorem_monotone}
Let $f:\real^n\to\real^n $ be an computable monotone function and let $z\in\zn$ be computably random. Then $f$ is differentiable at $z$.

\begin{proof}
\n Define $g=(f+I)^{-1}$, then $g$ is a computable Lipschitz function with $\Lip(g)\le1$.

\n Let $y=f(z)+z$ so that $g(y)=z$. By Lemma \ref{no_random_lemma}, $y$ is computably random and hence $g$ is differentiable at $y$ and by Theorem \ref{critical_point_theorem} $g'(y)$ is invertible. Hence, by Proposition \ref{inverse_diff_proposition}, $f$ is differentiable at $z$. 

\end{proof}
\end{theorem}

\section{Monotone transfer maps}\label{back_direction}

Suppose $z\in\R^n$ is not computably random and we want to exhibit a computable monotone function $f:\R^n\to\R^n$ that is not differentiable at $z$. Let us first overview how this problem has been resolved in the case when $n=1$. 

\begin{example}[On the real line]\label{ex:real}
Suppose $Z$ is the binary expansion of $z$.
We start with a martingale $M$ (we may assume it has the saving property) that succeeds on the $Z$ and define a computable measure on the real line by $\mu_M(\cylinder\sigma)=M(\sigma)\cdot 2^{-|\sigma|}$. Then the cumulative distribution of $\mu_M$,  $f=\cdf \mu_M$, is not differentiable at $z$. 
\end{example}

Before proceeding to generalize this construction in $\R^n$, we need to review some basic notions from the area known as \emph{optimal transport}.

\subsection{Optimal transportation}

Let $\mu,\nu$ be probability measures on $\R^n$. A probability measure $\pi$ on $\R^n\times \R^n$ is said to have \emph{marginals} $\mu$ and $\nu$ when the following holds for all measurable $A,B\subseteq \R^n$: 
\[
\pi\left[A\times \R^n\right]=\mu[A], \text{ and } \pi\left[\R^n\times B\right]=\nu[B].
\]

Let $\Pi(\mu,\nu)$ denote the set of all probability measures on $\R^n\times \R^n$ whose marginals are $\mu$ and $\nu$. Note that this set is always nonempty. For a given \emph{cost function} $c:\R^n\times\R^n\to \R$ and $\pi\in\Pi(\mu,\nu)$, define the \emph{total transportation cost} $I_c[\pi]$ as 
\[
I_c[\pi]=\int_{\R^n\times \R^n} c(x,y)~d\pi(x,y).
\]

The \emph{optimal transportation cost} between $\mu$ and $\nu$ is the value
\[
\I_c(\mu,\nu)=\inf_{\pi\in\Pi(\mu,\nu)} I_c[\pi].
\]

Let $T:\R^n\to\R^n$ be a map. We say $T$ is a \emph{transport map}, or that $T$ \emph{transports} $\mu$ onto $\nu$ (in symbols, $\nu=T\#\mu$), if for all measurable $A$, $\lambda(A)=\mu(T^{-1}(A))$.

Elements of $\Pi(\mu,\nu)$ are called \emph{transference plans}. We are interested in transference plans induced by measurable maps, that is, plans of the form $\pi_T=(I\times T)\#\mu\in\Pi(\mu,\nu)$ where $T:\R^n\to\R^n$ is a measurable map. The total transportation cost associated with a transport map $T$ is 
\[
I_c[T]=I_c[\pi_T]=\int_{\R^n}c(x,T(x))d\mu(x).
\]
A transport map $T$ for which the cost is optimal, that is for which $I_c[\pi_T]=\I_c(\mu,\nu)$, is called an \emph{optimal transport map}. The problem of minimizing $I_c[T]$ over the set of all transfer maps is known as \emph{Monge's optimal transportation problem}.

Let $U\subseteq \R^n$ be compact. For a function $f:U\to \R$, define its \emph{convex conjugate} $f^*$ by 
\[
f^*(y)=\sup_{x\in U}[xy-f(x)].
\] 

Since $U$ is assumed to be compact, $f^*$ is computable when $f$ is.

The following important result lies at the heart of our construction.

\begin{theorem}[Brenier's theorem, cf. Theorem 2.12 in \cite{Villani:03}]\label{thm:Brenier}
Let $\mu,\nu$ be probability measures on $\R^n$. Suppose $\mu$ is absolutely continuous (with respect to the Lebesgue measure) and the following holds:
\[
\int_{\R^n}\frac{|x|^2}{2}d\mu(x)+\int_{\R^n}\frac{|y|^2}{2}d\nu(y)<\infty.
\]

Then there exists a convex function $\phi$ such that $\nabla \phi \#\mu=\nu$. Moreover, $\nabla \phi$ is the unique (i.e. uniquely determined $\mu$-almost everywhere) gradient of a convex function which pushes $\mu$ forward to $\nu$.

Furthermore, if $\nu$ absolutely continuous, then, for $\mu$-almost all $x$ and for $\nu$-almost all $y$,
\[
\nabla\phi^*\circ\nabla\phi(x)=x,~~~~~~~~~ \nabla\phi\circ\nabla\phi^*(y)=y,
\]

and $\nabla\phi^*$ is the ($\nu$-almost everywhere) unique gradient of a convex function which pushes $\nu$ forward to $\mu$, and also the solution of the Monge problem for transporting $\nu$ onto $\mu$ with a quadratic cost function.

\end{theorem}

\begin{remark}
The above result is known to hold not only for absolutely continuous measures, but more general results are not needed in this paper.
\end{remark}

Now we are ready to review the Example \ref{ex:real} in the context of optimal transport theory.

\subsection{The main idea}

Let $f=\cdf \mu_M$ be the function from the example.
Note that $f$ is a transport map from $\mu_M$ to the Lebesgue measure $\lambda$ (that is, $\lambda=f\#\mu_M$). In fact, by the optimal transportation theorem for a quadratic cost of $\R$ (see Theorem 2.18 in \cite{Villani:03}), $f$ is the (unique) optimal transport map from $\mu_M$ to $\lambda$. Unlike in higher dimensions, on the real line, the form of the optimal transport map is known and in our case (a special case of transporting $\mu_M$ onto $\lambda$), the function $f$ is the optimal one.

Note that the derivative $D_\lambda \mu_M(z)$ of $\mu_M$ with respect to the Lebesgue measure  does not exist. Intuitively, $\mu_M$ oscillates around $z$ and, correspondingly, the transport map is not differentiable at $z$.

\subsection{Wasserstein metrics}

Given a Polish metric space $(X,d)$, the set $\P(X)$ of Borel probability measures over $X$ endowed with the weak topology is a Polish  space. 

Suppose $(X,d,\seq \alpha)$ is a computable metric space where $d$ is bounded. Let $\seq \delta$ be an effective enumeration of those elements of $\P(X)$ which are concentrated on finite subsets of special points and assign rational values to them. Let $\pi$ be the Prokhorov metric on $\P(X)$, then $(\P(X),\pi, \seq \delta)$ is a computable metric space compatible with the weak topology on $\P(X)$. Following \cite{Gacs:05} and \cite{Hoyrup.Rojas:09}, we define computable measures as computable elements of $(\P(X),\pi, \seq \delta)$. 

For $p\in\N$ with $p\ge 1$, define the cost function $c_p$ by $c_p(x,y)=d(x,y)^p$. 
For $\mu,\nu\in\P(X)$, define the \emph{Wasserstein metric of order $p$} by $$W_p(\mu,\nu)=\I_p(\mu,\nu)^{1/p}$$ where $\I_p$ is the optimal transport cost between $\mu$ and $\nu$ with respect to $c_p$.   It is known that $W_p$ metrizes the weak topology on $\P(X)$. Furthermore, $W_1$ is computable and it is\emph{ computably equivalent} to $\pi$ \cite{Hoyrup.Rojas:09}. That is, given a Cauchy name of $\mu$ with respect to $\pi$, it is possible to compute a Cauchy name  of $\mu$ with respect to $W_1$ and vice versa. Since we are mainly concerned with the quadratic cost, we need to prove an analogous result for $p>1$. 

\begin{proposition}
Let $(X,d,\seq \alpha)$ be a computable metric space where $d$ is bounded.
 Let $p> 1$. Then $W_p$ is computably equivalent to $W_{1}$ (and hence to $\pi$).
\end{proposition}

\begin{proof}
Firstly, note that $W_p(\delta_i,\delta_j)$ is computable uniformly in $i,j,p$. This is due to the fact, that computing  $W_p$ between discrete measures is a linear programming problem, for which there are algorithms available. (\textbf{TODO: some refs})

\n It is known (see 7.1.2 in \cite{Villani:03}) that the following inequalities hold: 
\begin{align}\label{label_2}
W_1\le W_p\le W_1^{1/p}\text{diam}(X)^{1-1/p}. 
\end{align}

Fix a computable real $D$ with $D\ge\text{diam}(X)$.
Let $\seq \mu$ be Cauchy names of $\mu$ with respect to $\pi$. Let $i,j\in\N$. Using \ref{label_2} and the triangle inequality we have
\begin{align*}
W_p(\mu,\delta_j)\le W_p(\delta_j,\mu_i)+W_p(\mu_i,\mu)\le \\
D\cdot W_1^{1/p}(\mu,\mu_i)+W_p(\delta_j,\mu_i).
\end{align*}
This shows that we can effectively find a Cauchy name with respect to $W_p$ given a Cauchy name with respect to $\pi$.

For the other direction, suppose $\seq{\hat\mu}$ is a Cauchy name of $\mu$ with respect to $W_p$. Then 
\begin{align*}
W_1(\mu,\delta_j)\le W_1(\delta_j,\hat\mu_i)+W_1(\hat\mu_i,\mu)\le \\
W_1(\delta_j,\hat\mu_i)+W_p(\hat\mu_i,\mu).
\end{align*}
The required result follows.
\end{proof}

\begin{corollary} $\I_p(\mu,\nu)$ is computable uniformly in $\mu,\nu$ and $p$.
\end{corollary}

\subsection{An effective version of Brenier's theorem}

\begin{theorem}\label{effective_brenier_1}
Let $\mu,\nu$ be absolutely continuous computable probability measures on $\R^n$ with $\text{supp}(\mu)=\zn$. There exists a computable convex function $\phi:\R^n\to\R$ such that $\nabla\phi$ is the optimal transport map from $\mu$ to $\nu$.

\begin{proof}

From Theorem \ref{thm:Brenier} we know that there a unique convex function $\phi$ such that $\phi(0)=0$ and $\nabla\phi$ is the optimal transformation map from $\mu$ to $\lambda$. Since it doesn't matter how $\phi$ is defined outside of $\text{supp}(\mu)$, we may assume $\phi$ is Lipschitz. Pick some rational $K\in\Q$ so that $K>\Lip(\phi)$ and consider the subspace 
\[
L_0(K)=\{f\in C\zn: \Lip{f}\le K\text{ and } f(0)=0\}.
\]

By Arzela-Ascoli theorem, $L_0(K)$ is a compact subspace of $C\zn$ (the space of real valued continuous functions endowed with the supremum metric) containing $\phi $. Moreover, since the support of $\mu$  is equal to $\zn$ and $\nabla\phi$ is uniquely determined $\mu-$a.e.,  $\phi$ is the only function in $L_0(K)$ for which $\nabla\phi$ is optimal. 

Recall, that a function $f:\R^n\to \R$ is called \emph{piecewise affine} if there exists a finite set of affine functions $f_i(x)=A_i\cdot x+b_i$, $i=1,\dots,k$, such that the inclusion $f(x)\in\{f_i(x),\dots,f_k(x)\}$ holds for all $x$. The functions $f_i$ are called \emph{selection functions}. The set of pairs $(A_i,b_i)$ is called a collection of \emph{matrix-vector pairs} corresponding to $f$. If $f$ is a piecewise affine function and $(A_i,b_i)$ for $i\le k$ are the corresponding matrix-vector pairs, there exists a finite number of index sets $M_1,\dots, M_l\subseteq \{1,\dots,k\}$ such that 

\[
f(x)=\max_{1\le i\le l}\min_{j\in M_i}A_j\cdot x+b_j\text{ for all }x.
\]

For every $i\in\N^+$, let $\D_i^n$ denote the set of points in $\R^n$ with all coordinates of  the form $k2^{-i}$ for some integer $k$.

It is known that piecewise affine functions are dense in $L_0(K)$. For $k\in\N^+$, let $\Gamma_k$ be the (finite) set of piecewise affine functions $f$ such that 
\begin{enumerate}
\item all of its matrix-vector pairs belong to $\D^n_k\times D^1_k$,
\item  $\Lip(f)\le K$, and
\item  $f(0)=0$.
\end{enumerate}

 Note that it is possible to effectively enumerate elements of $\Gamma_k$ uniformly in $k$.  Let $\seq\gamma$ an effective enumeration of $\bigcup_k\Gamma_k$. It is dense in $L_0(K)$.

\begin{claim} $L=\left(L_0(K),\|\cdot\|_\infty,\seq\gamma\right)$ is an effectively compact computable metric space.
\end{claim}
\begin{proof}
$L$ is clearly a computable metric space. To show that it is effectively compact, fix $i\in\N^+$. Then $\{B(\gamma; 2^{-i})~|~\gamma\in \Gamma_i\}$ form a finite open cover of $L$. 
\end{proof}

\n For any $f,g:\R^n\to \R$ let 
\[
J(f,g)=\int_{\R^n}f(x)~d\mu(x)+\int_{\R^n}g(x)~d\lambda(x).
\]

\n We know that $J(\phi,\phi^*)=\I_2(\mu,\lambda)$ (see the proof of Theorem 2.12 in \cite{Villani:03}).

Define 
\[
S=\{f\in L_0(K)~|~ J(f,f^*)=\I_2(\mu,\lambda)\}.
\]

The condition $J(f,f^*)=\I_2(\mu,\lambda)$ is (uniformly) computable in $f$ and $\mu$. Hence $S$ is a $\arpi01$ subset. Since $\phi$ is uniquely defined $\mu-$a.e.,  $S$ contains only one element - $\phi$.

Let us show that effective compactness of $L_0(K)$ guarantees computability of $\phi$. The set $A_S$ of basic open balls disjoint from $S$ is recursively enumerable. As we have shown, for a given $j\in\N$, we can find a finite cover of $L_0(K)$ by basic open $2^{-j}-$balls.  Let us denote such covers  $I_j$. 
Fix $i\in\N$. Enumerate elements of  $A_S$ and elements of those $I_j$, where $j\ge i,$ until all the balls in $I_{i+3}$ that has not been enumerated so far have centers at most $2^{-i-1}$ from each other. Let $\gamma$ be one such center. Then the basic open ball $B(\gamma; 2^{-i})$ contains $S$. Therefore, $\phi$ is computable.

\end{proof}
\end{theorem}

\subsection{Application to computable randomness on $\R^n$}

In this subsection we prove the following converse to the Theorem \ref{theorem_monotone}.
\begin{theorem}\label{thm:monotone}
Suppose $z\in\R^n$ is not computably random. Then there exists a computable monotone function $f:\R^n\to\R^n$ not differentiable at $z$.
\end{theorem}

We may assume that every coordinate of $z$ is computably random. For suppose $z_i$ is not computably random for some $i$. There exists a computable monotone function $g:\R\to\R$ not differentiable at $z_i$. Then the function $(x_1,\dots,x_n)\mapsto (g(x_1),\dots,g(x_n))$ is a computable monotone function from $\R^n$ to $\R^n$ not differentiable at $z$.

Since $z$ is not computably random, there exists an absolutely continuous computable probability measure $\mu$ on $\R^n$ such that $D_\lambda\mu(z)$ does not exist. This follows from Theorem 5.3(4) \cite{Rute:14:1} and the fact that all coordinates of $z$ are computably random (and hence non-dyadic). Without loss of generality we may assume that $\mu$ is supported on $\zn$.
The following classical result is needed to show that the optimal transport map (from $\mu$ onto $\lambda$) is not differentiable at $z$.

\begin{theorem}[Jacobian theorem for monotone maps, cf. Theorem A.2 in \cite{McCann:97}]\label{thm:A2}
Let $\phi$ be a convex function on $\R^n$ and suppose it is twice differentiable at $x\in\R^n$. Then 

\[
\lim_{r\to0} \frac{\lambda\left(\partial\phi (B_r(x))\right)}{\lambda\left(B_r(x)\right)} = \det D_A^2 \phi(x).
\]

\end{theorem}

Theorem \ref{thm:Brenier} shows that there exists a unique monotone function $f=\nabla \phi$, such that $f\#\mu=\lambda$, where $\phi$ is some convex function (obviously, not unique). Since $\mu$ is absolutely continuous, by Theorem 2.12 (iv) and Lemma 4.6 from \cite{Villani:03}, $\lambda\left(\partial \phi(A)\right)=\lambda\left(\nabla \phi(A)\right)$ for all Borel $A\subseteq \R^n$. Hence 

\[
\lim_{r\to0} \frac{\lambda\left(\partial\phi (B_r(x))\right)}{\lambda\left(B_r(x)\right)}=\lim_{r\to0} \frac{\lambda\left(\nabla\phi (B_r(x))\right)}{\lambda\left(B_r(x)\right)}=D_\lambda \mu(x)
\] for all $x$. By Theorem \ref{thm:A2}, $\nabla\phi$ is not differentiable at $z$.

To complete the proof of Theorem \ref{thm:monotone}, we need to show that $f$ is a computable function. Theorem \ref{effective_brenier_1} shows that there exists a computable convex function $\phi$ such that $f=\nabla\phi$. In the following subsections, we prove that under some additional assumptions on $\mu$, $\phi$ is actually $C^{1,\alpha}$ and thus $\nabla \phi$ is computable.

\subsubsection{Computability of $\nabla \phi$.}

For a given martingale $M$, we define a computable probability measure $\mu_M$ on $\zn$ by 
\[
\mu_M(\cylinder \sigma)=\lambda(\cylinder \sigma)\cdot M(\sigma) \text{ for all }\sigma \text{ with } |\sigma|=ns\text{ for some }s.
\]

\begin{lemma}

Let $M$ be a computable martingale and let $\mu_M$ be the corresponding probability measure on $\zn$. Let $z\in\zn$  and let $Z$ be the binary expansion of $z$. 
Suppose, the following two conditions hold:
\begin{enumerate}
\item[(P1)] the measure $\mu_M$ is absolutely continuous, not differentiable at $z$, and
\item[(P2)] $0<M(\sigma)<C$ for some fixed $C$ and all $\sigma$.
\end{enumerate}
Then the optimal transport map from $\mu_M$ to $\lambda$ is computable.
\end{lemma}
\begin{proof}
We know that there is a computable convex function $\phi$ such that $\nabla \phi$ is the optimal transfer map of $\mu$ onto $\lambda$. Define $h(x)=D_\lambda\mu_M(x)$. Then $\phi$ is an Aleksandrov solution of the following instance of the Monge-Amp\'ere equation: 
\[
\det D^2 f=h.
\]
Since $h$ is bounded away from both $0$ and $\infty$, by Theorem 4.13 in \cite{Villani:03}, $\phi$ is $C^{1,\alpha}$ for some $\alpha>0$. Since $\nabla\phi$ is a.e. computable and H\"older continuous, it must be computable.
\end{proof}

To complete the proof we will describe a construction of a martingale $M$ satisfying the conditions of the previous lemma.

\subsubsection{Construction of the martingale $M$.}

\begin{lemma}
Suppose $z\in\zn$ is not computably random. There does exist a computable martingale $M$ that satisfies the (P1) and (P2) conditions. 
\end{lemma}
\begin{proof}

Again, let $Z$ denote binary expansion of $z$. We will modify the construction described in the proof of Theorem 4.2 in \cite{Freer.Kjos.ea:14}. 

The martingale $B$ constructed in the proof of Theorem 4.2 has the following properties relevant to us:
\begin{enumerate}
\item $1\le B(\sigma)\le 4$ for all $\sigma$. Let's call $1$ and $4$ the \emph{bounding constants of $B$}.  
\item $B$ has two distinct \emph{phases}: the \emph{up phase} (where it increases its capital) and the \emph{down phase} (where it decreases the capital). While $B$ is in the up phase, its capital reaches the value of $3$, and while $B$ is in the down phase, its capital reaches the value of $2$. 
The construction guarantees that the capital of $B$ along $Z$ oscillates - that is, $B$ alternates between two phases infinitely often. Let's call $2$ and $3$ the \emph{oscillation constants of $B$}.
\end{enumerate} 

The martingale $B$ satisfies the property $P2$, but not necessarily the $P1$ property since $D_\lambda\mu_B(z)$ might still exist despite oscillations of $B$ on the binary expansion of $z$.

The construction can be modified to define a variant of $B$ (let's call it $M$), that satisfies both properties.

Firstly, note that both bounding and oscillation constants are flexible. In particular, we may assume that oscillation constants of $M$ are some rational numbers $p>q>1$ and its bounding constants are $q-1$ and $p+1$. Also note that when the capital of $M$ reaches the value $>p$ ($<q$), $M$ can maintain its capital at the  $\ge p$ ($\le q$) level for as long as needed. 

Before explaining what conditions on $M$ guarantee both P1 and P2 hold, let us define some notation and recall one geometric fact about basic dyadic cubes.

\n Let $n\in \N$. Define $\D_n$ to be the set of all (basic) dyadic cubes in $\real^n$. Let $T_n=\{0,1/3,2/3\}^n$. For every $t\in T_n$ define $\D^t_n=\left\{Q+t:Q\in \D_n \right\}$ and denote by $Z^t$ the binary expansion of $z+t$.

\n The following fact is known as the ``one third trick''.

\begin{fact} There is a universal constant $K_n>0$ such that for any ball $B\subset\real^n$ with radius $r < 1/3$, there is $t\in T_n$ and a cube $Q\in\D_n^t$ containing $B$ whose radius is no more than $K_n\cdot r$.
\end{fact}

(\textbf{TODO: replace with a variant of Theorem 3.8 from Olli Tappiola's thesis and name the constants appropriately})

A simple consequence of this fact is that there exists $k_n\in\N$, such that for any sufficiently small $r$ the following holds 
\begin{align}\label{l1}
\cylinder {\px {Z^t}{sn}}\subseteq B_r(z+t)\subseteq \cylinder {\px {Z^t}{sn-k_n}}
\end{align}

for some $t\in T_n$ and $s\in\N$.

Thus, for some particular value of $t$, ($\ref{l1}$) holds for infinitely many $s$. Since the operation $f\mapsto f+t$ preserves monotonicity and computability, we may assume ($\ref{l1}$) holds infinitely often for $t=0$. 

Suppose ($\ref{l1}$) holds for some $s,r$ and $M(\px {Z}{sn-k_n})\le q$. Let $D_2=\cylinder {\px {Z}{sn}},B= B_r(z)$ and $ D_1= \cylinder {\px {Z}{sn-k_n}}$. Note that 
\[
2^{-k_n}\le \frac{\lambda(B)}{\lambda(D_1)}\le 1\text{ and } 2^{-k_n} \le \frac{\lambda(D_2)}{\lambda(B)}\le 1.
\]

Then we have
\[
\frac {\mu(B)}{\lambda(B)}=\frac{\mu(D_1)-\mu(D_1\setminus B)}{\lambda(D_1)-\lambda(D_1\setminus B)}\le\\
\frac{\mu(D_1)}{\lambda(D_2)}=2^{k_n}\cdot \frac{\mu(D_1)}{\lambda(D_1)}\le 2^{k_n}\cdot q.
\]

Analogously, assuming $M(\px {Z}{sn})\ge p$, we get 
\[
\frac {\mu(B)}{\lambda(B)}=\frac{\mu(D_2)+\mu(B\setminus D_2)}{\lambda(D_1)}\ge\\
\frac{\mu(D_2)}{\lambda(D_1)}=2^{-k_n}\cdot \frac{\mu(D_2)}{\lambda(D_2)}\ge 2^{-k_n}\cdot p.
\]

Hence the following three conditions  imply that $D_\lambda \mu(z)$ does not exist:
\begin{enumerate}
\item $2^{-k_n}\cdot p - 2^{k_n}\cdot q>0$, 
\item $M(\px {Z}{sn})\ge p$ and $\cylinder {\px {Z}{sn}}\subseteq B_r(z)\subseteq \cylinder {\px {Z}{sn-k_n}}$ hold for infinitely many $s$, and
\item $M(\px {Z}{sn-k_n})\le q$ and $\cylinder {\px {Z}{sn}}\subseteq B_r(z)\subseteq \cylinder {\px {Z}{sn-k_n}}$ hold for infinitely many $s$.
\end{enumerate}

The first condition is trivially met by setting $p$ and $q$ appropriately. 
To ensure the second condition is met, the martingale $M$, once it is in the up phase and its capital is $> p$, waits (maintaining the value of its capital $> p$) until it is clear that there are two basic dyadic cubes satisfying \ref{l1} (this can be done by checking that the distance between boundaries of two dyadic cubes is large enough). 

The third condition can be dealt with in an analogous manner.
\end{proof}

\def\cprime{$'$}

\bibliographystyle{plain}

\end{document}